\documentclass{amsart}

\usepackage{amsthm}

\newtheorem{theorem}{Theorem}[section]
\newtheorem{corollary}{Corollary}[theorem]
\newtheorem{lemma}[theorem]{Lemma}
\newtheorem{proposition}[theorem]{Proposition}

\theoremstyle{definition}
\newtheorem{definition}[theorem]{Definition}
\newtheorem{example}[theorem]{Example}

\theoremstyle{remark}
\newtheorem*{remark}{Remark}

\numberwithin{equation}{section}

\newcommand{\R}{\mathbb{R}}  
\newcommand{\Z}{\mathbb{Z}}  
\newcommand{\N}{\mathbb{N}}  
\newcommand{\Nb}{\mathfrak{N}}  
\newcommand{\Orth}{{\mathbb O}}
\newcommand{\interior}{\mathop{\it int}}
\newcommand{\relint}{\mathop{\it relint}}

\newcommand{\sd}{\mathop{\rm Sd}}
\newcommand{\cd}{\mathop{\rm Cd}}
\newcommand{\conv}{\mathop{\rm conv}}
\newcommand{\dvee}{\mathord{\mathrm v}}

\begin{document}

\title{Neighbors, Generic Sets and Scarf-Buchberger Hypersurfaces}


\author{James J.~Madden}
\address{Louisiana State University, Baton Rouge}
\curraddr{}
\email{madden@math.lsu.edu, jamesjmadden@gmail.com}
\thanks{}

\author{Trevor McGuire}
\address{North Dakota State University, Fargo}
\curraddr{}
\email{trevor.mcguire@ndsu.edu, trevor.e.mcguire@gmail.com}
\thanks{}

\subjclass[2010]{Primary }

\keywords{}

\date{}

\dedicatory{}

\begin{abstract} 
  The present paper is motivated by the need to generalize the construction of the Scarf complex in order to give combinatorial resolutions of a much broader class of modules than just the monomial ideals.  For any subset $A\subseteq \R^n$, let $\Nb(A)$ denote the collection of all subsets $B\subseteq A$ such that  there is no $a\in A$ that is strictly less than the supremum of $B$ in all coordinates.      We show that if $A\subseteq \Z^n$ is generic (in a sense appropriate for this context), then $\Nb(A)$ is a locally finite simplicial complex.  Moreover, if $A$ is generic, then the barycentric subdivision of $\Nb(A)$ is equivalent to a triangulation of a PL hypersurface in $\R^n$.  This gives us natural generalizations of the notions of ``staircase surface'' and ``Buchberger graph,'' see \cite[ch.~3]{MS}, to arbitrary dimension.  (This seems to be a new result, even in the well-studied case that $A$ is a finite subset of $\N^n$.)  We give examples that show that when $A$ is infinite, $\Nb(A)$ may have complicated topology, but if there are at most finitely many elements of $A$ below any given $b\in \R^n$, then $\Nb(A)$ is locally contractible. $\Nb(A)$ can therefore be used to construct locally finite free resolutions of  sub-$k[\N^n]$-modules of the group algebra  $k[\R^n]$ ($k$ is a field).  We prove various additional facts about the structure of $\Nb(A)$.\end{abstract}

\maketitle

\section{Introduction} 
In a study of integer programming \cite{Sc},  H.~Scarf introduced a certain simplicial complex constructed from a set $A$ of points in $\R^n$.    Bayer, Peeva and Sturmfels \cite{BPS} discovered a remarkable application Scarf's construction in algebra.  They defined the Scarf complex of a monomial ideal to be the complex obtained by Scarf's method when $A$ consists of the exponent vectors of a minimal monomial generating set of the ideal.  If the ideal is generic, then the Scarf complex supports a minimal free resolution.  

Let $k$ be a field and let $S=k[\N^n]$ be the polynomial algebra in $n$ variables.   In \cite{M}, the second author described  minimal free resolutions of certain ideals of $S$ generated by monomials and binomials.  As a key step in this work, it was necessary to generalize results from  \cite{BPS} to the case in which $A$ is an infinite subset of  $k[\Z^n]$.  Some generalizations had already been established in the case that $A$ is a subgroup of $\Z^n$, see \cite[ch.~9]{MS}, but further generalizations were needed.     Subsequent to \cite{M}, we sought to generalize and streamline some of the results there.  We found it convenient to develop the combinatorial foundations independently from the algebraic applications.  The present paper reports on this work.

Let $A$ a generic antichain in ${\R}^n$.  (See section 3 for definitions.)  Following the approach of Scarf \cite{Sc}, we construct a simplicial complex $\Nb(A)$ and a PL hypersurface $\partial dmA$ embedded in ${\R}^n$.      We then introduce a canonical triangulation $\cd \Nb(A)$ of  $\partial dmA$ and prove that it equivalent to the barycentric subdivision of $\Nb(A)$.  In case $A$ is the set of exponent vectors of a minimal generating set of a generic monomial ideal, $\Nb(A)$ is the Scarf complex of that ideal.  If $n=3$, the Buchberger graph \cite[\S 3.3]{MS} is contained in the one-skeleton of $\cd \Nb(A)$.  In a future paper, we apply the facts proved here about $\cd \Nb(A)$ to describe minimal free resolutions of monomial sub-$S$-modules of the Laurent algebra  $k[X_1^{\pm 1},\ldots, X_n^{\pm 1}]$ and---by using the equivariant methods described in \cite[ch.~9]{MS} and other tools---we will derive  combinatorial descriptions of resolutions of ideals of $k[\N^n]$ generated by monomials and binomials.  

The main contributions of the present paper are as follows. First, we recall the conceptual framework introduced in \cite{Sc} for the definition of the complex $\Nb(A)$ of $A$-free convex bodies.  Here $A$ may be any subset of $\R^n$.  We also recall and elaborate Scarf's geometric interpretation of ``generic'' for such sets.  Second, we prove that if $A\subseteq \Z^n$ is generic, then $\Nb(A)$ is locally finite (in the sense that every vertex is in at most finitely many simplices).  Note that \cite[Theorem 1.8]{Sc}---which is also, \cite[Theorem 9.14]{MS}---shows that if $A$ is a lattice then $\Nb(A)$ is locally finite, with no genericity assumption.  However, for the applications we have in mind, we need to have local finiteness when $A$ is not a lattice.  Third, we prove that the barycentric subdivision of $\Nb(A)$ for $A$ generic is a triangulation of a subset of a PL hypersurface in $\R^n$.  This gives meaning to the remark of \cite{BHS} that $\Nb(A)$ is an ``intricate folding of $\R^n$ into itself,'' and in fact indicates what the folding is.  Fourth, we show that when $A$ is infinite, $\Nb(A)$ may have complicated topology even if it is generic, but if $A$ is locally finite (in the sense that there are only finitely many vertices in any principle down-set) then $\Nb(A)$ is locally contractible.  This means that $\Nb(A)$ can be used to construct locally finite free resolutions of  sub-$S$-modules of the group algebra  $k[\R^n]$, where $k$ is a field.  Finally, we extend a result of \cite{BSS}, showing that if $A^\ast$ is the augmentation of $A$ by ideal points at infinity, then, $\Nb(A^\ast)$ is a triangulation of the $n$-simplex. 

The concepts discussed in section 3, as well as their geometric  interpretation in terms of translates of $-\Orth^n$ come from \cite{Sc}.  Definition 3.1 of the present paper is the same as \cite[Definition 1.3]{Sc}, except for inverting the order, and much of section 3 is explicitly or implicitly in \cite{Sc}.  The idea for Definition 7.1 is in \cite[Definition 1.4]{Sc}, and the notation we use is from \cite{BSS}.   In much of his work, Scarf considers an $m\times n$ ($m>n$) matrix $M$, and studies the subgroup $A:= M\Z^n \subseteq \R^m$. (Scarf calls this matrix ``$A$'' not ``$M$,'' but this clashes with the notation we have chosen.)    In \cite{BSS}, for example, the authors assume that the set $A$ lies in a hyperplane  $H\subseteq \R^m$  perpendicular to a vector $0<<\lambda \in \R^m$.  This assures that $A$ is an antichain.  They make additional assumptions about $M$ that imply that $A$ is generic in the sense that we define in section 3.      By  ``convex body,'' Scarf and co-authors mean an intersection of some $b-\Orth^m$ with $H$.  In their terminology, a convex body is ``lattice-free'' if $b-\interior\Orth^m$ contains no points of $A$.  It is maximal if: for all $b'\geq b$, $b'-\interior\Orth^m$ lattice-free $\Rightarrow$ $b'=b$.  One important difference between the situation considered in \cite{BSS} and the present in that we do not require $A$ to lie in a hyperplane.  It need not be contained in any proper affine subspace of $\R^n$, or even any finite union of proper affine subspaces, in order for our methods to apply.
   
Olteanu and Welker \cite{OW} have studied the abstract simplicial complex $\Nb(A)$ (as defined in section 3, below) in the case that $A$ is an antichain in $\N^n$, showing, among other things, that it is always contractible.   We recapitulate, simplify and apply some of their work in section 6.  They define the Buchberger complex of a monomial ideal as follows (translated into our notation):

\begin{definition} Let $I\subseteq k[x_1,\ldots, x_n]$ be a monomial ideal with minimal generating set $G_I$. The {\it Buchberger complex\/} of $I$ is the collection of all subsets $B\subseteq G_I$ such that for every $g\in G_I$, there is at least one coordinate in which the exponent vector of $g$ equals or exceeds the exponent vector of $\mathrm{LCM}(B)$.\end{definition}

We would call this $\Nb(G_I)$.  The content of the present paper differs from \cite{OW} in the following respects.  First, in \cite{OW}, $G_I$ is always a finite antichain in $\N^n$, whereas the antichains $A$ that we consider are contained in $\R^n$ and may be infinite.  Second, \cite{OW} devotes considerable attention to $\Nb(G_I)$ without any genericity assumptions, but our attention is devoted exclusively to examining consequences of the generic assumption.  Third, Olteanu and Welker determine properties of the abstract simplicial complex $\Nb(G_I)$ and a related complex $P(G_I)$ (which we describe in section 5, below).  One of our main results is to exhibit a concrete realization of a barycentric subdivision of $\Nb(A)$ as a subcomplex of a piecewise linear hypersurface in $\R^n$.

\section{Notation for $\R^n$, $\Z^n$ and other posets}

This section explains the notation used throughout this paper.   We will use lower-case Roman letters to denote elements of $\R^n$. Subscripts, as in  $x_1, x_2,\ldots$, are used to distinguish between different elements of $\R^n$.  The $i^{th}$ component of $x$ is denoted $\pi_i(x)$.   We  employ the following notation for the coordinate-wise partial order on $\R^n$: 
\begin{align*}
x\leq y\quad:\,\Leftrightarrow\quad& \pi_1(x)\leq\pi_1(y)\;\&\;\cdots\;\&\;\pi_n(x)\leq\pi_n(y),\\
x< y\quad:\,\Leftrightarrow\quad&\hbox{$x\leq y$ and $\pi_i(x)<\pi_i(y)$ for some $i\in\{1,2,\ldots,n\}$},\\
x<< y\quad:\,\Leftrightarrow\quad&\pi_1(x)<\pi_1(y)\;\&\;\cdots\;\&\;\pi_n(x)<\pi_n(y).
\end{align*}
Under this order, $\R^n$ is a distributive lattice.  The supremum and infimum are denoted   $x\vee y$ and $x\wedge y$, respectively; $x^+:=x\vee 0$ and  $x^-:=-x\vee 0$.  Every subset $X$ of $\R^n$ with an upper bound (lower bound) in $\R^n$ has a least upper bound, denoted $\vee X$ (greatest lower bound, denoted
$\wedge X$) in $\R^n$.   A subset  $X$ of $\R^n$  is called an {\it antichain} if for all $x, y\in X$, $x\leq y\;\Rightarrow\;x= y$. $X$ is called an {\it up-set in $\R^n$\/} (a {\it down-set in $\R^n$\/}) if $x\in X$ and $x\leq y\in \R^n$ ($x\geq y\in \R^n$)  implies $y\in X$.  

The {\it positive orthant\/} of $\R^n$, denoted  $\Orth^n$, is the set $\{\,x\in \R^n\mid 0\leq x\,\}=[0, +\infty)^n$.  If $X\subseteq \R^n$,  $X+\Orth^n$ ($X-\Orth^n$) is the smallest up-set in $\R^n$ (smallest down-set in $\R^n$) containing $X$.  The {\it bonnet over $X$\/} is the smallest down-set that contains $X$ and is closed under the operation of taking suprema, i.e., it is the lattice ideal of $\R^n$ generated by $X$.   If $X$ has an upper bound, then the bonnet over $X$ is $(\vee X)-\Orth^n$.  

One forms the {\it order completion\/} of $\R$ by adjoining elements $-\infty$ and $+\infty$ to $\R$ and ordering the result such that $-\infty< x<y< +\infty$ for all $x<y$ in $\R$.  The result is denoted $[-\infty, +\infty]$.   {\it Every\/} subset of $[-\infty, +\infty]^n$ has a supremum in $[-\infty, +\infty]^n$.  We denote this  $\vee X$, relying on the context to make it clear that we are taking the supremum in $[-\infty, +\infty]^n$.  If $X\subseteq \R^n$ and $X$ is bounded in $\R^n$, then obviously the supremum of $X$ in $[-\infty, +\infty]^n$ is the same as the supremum of $X$ in $\R^n$.    The {\it bonnet over $X$ in $[-\infty, +\infty]^n$\/} is the down-set of $\vee X$.   

The {\it interior of $\Orth^n$\/}, denoted
$\interior\Orth^n$, is $\{\,x\in \R^n\mid x>>0\,\}$.  Let $J\subseteq \{1,2,\ldots,n\,\}$.  The {\it $J^{th}$ face of  $\Orth^n$\/} is 
$$\Orth^n_J:=\{\,x\in \Orth^n\mid \hbox{$\pi_j(z)=0$ for all $j\in J$}\,\}.$$  Note that $\Orth^n_\emptyset = \Orth^n$ and  $\Orth^n_{\{1,\ldots,n\}} = \{0\}$.  Moreover, $$\Orth^n_{J\cup K}=\Orth^n_J\cap \Orth^n_K,\eqno{(1.1)}$$ and $\Orth^n_{J\cap K}$ is the smallest convex cone containing both $\Orth^n_J$ and $\Orth^n_K$.    The cardinality of $J$ is the {\it codimension of $\Orth^n_J$\/}.   
The {\it relative interior of $\Orth^n_J$\/}, denoted
$\relint\Orth^n_J$, is $\{\,x\in \Orth^n_J\mid \hbox{$\pi_i(x)>0$ for all $i\not\in J$}\,\}$.    
 We apply similar terminology to translates of $-\Orth^n$.  The {\it interior of $x-\Orth^n$\/}  is $x-\interior\Orth^n$.  The $J^{th}$ face of $x-\Orth^n$ is $x-\Orth^n_J$, and its relative interior is $x-\relint \Orth^n_J$.
 
As a sub-poset of $\R^n$, $\Z^n$ is closed under $\vee$ and $\wedge$: if $X\subseteq \Z^n$, then $\vee X$, if it exists in $\R^n$, lies in $\Z^n$.  All of the previous notation restricts to $\Z^n$ by intersection.  For example,  the {\it positive orthant of $\Z^n$} is  $\N^n:=  \{\,\alpha\in \Z^n\mid 0\leq \alpha\,\}=\Z^n\cap \Orth^n$.      When dealing with elements of $\Z^n$ we sometimes use Greek rather than Roman letters.

If $Q\subseteq P$ are posets, and $p\in P$ then $Q_{\leq p} :=\{\,q\in Q\mid q\leq p\,\}$ and $Q_{< p} :=\{\,q\in Q\mid q< p\,\}$.  We also use interval notation, e.g., for $p_1, p_2\in P$, $[p_1, p_2)_Q:= \{\,q\in Q\mid p_1\leq q< p_2\,\}$. When it is obvious what set is being referred to, we may omit the subscript.    Note that $\Orth^n = \R^n_{\geq 0}$ and $\interior\Orth^n = \R^n_{>> 0}$.

\section{Neighbors and generic sets}

Most of what we present in this section is a rephrasing of material from \cite{Sc}.  Lemmas \ref{genericlemma} and \ref{codimlemma}  elaborate on that material. Our terminology 

\begin{definition} Let $A, B\subseteq \R^n$.  We say that $B$ is {\it $A$-neighborly\/} if: $a)$ $B$ has an upper bound, and hence a least upper bound $\vee B$,  and $b)$ $\vee B-\interior\Orth^n$ contains no elements of $A$, i.e., there are no points of $A$ in the interior of the bonnet over $B$.  \end{definition} 

Typically, we are only interested in cases were $B\subseteq A$, but the definition makes sense without this assumption.  We say that $y,y'\in \R^n$ are $A$-neighbors if  $\{y,y'\}$ is $A$-neighborly.  The set of elements of $A$ that are $A$-neighbors of $y\in \R^n$ is denoted  ${\rm nbr}_A(y)$.

\begin{definition} The set of all finite $A$-neighborly subsets of $A$ is denoted $\Nb(A)$. The collection of all $A$-neighborly subsets of $A$ containing $d+1$ elements is denoted $N_d(A)$.\end{definition}

 If $B'\subseteq B$, then the bonnet over $B'$ is contained in the bonnet over $B$ and the interior of the bonnet over $B'$ is contained in the interior of the bonnet over $B$.  Hence, if $B$ is finite and $A$-neighborly, then any subset of $B$ is $A$-neighborly.  Accordingly, we have the following:

\begin{lemma}  $\Nb(A)$ is an abstract simplicial complex.\end{lemma}

\noindent  Note that $N_{-1}(A)=\{\emptyset\}$.  If $A$ is an antichain, then $N_0(A)$ is equal to the set of all singleton subsets of $A$, but this can happen even when $A$ is not an antichain.  For example, let $A=\R\times\{0\}\subseteq \R^2$.


\begin{definition} We say that $A$ is {\it generic\/} if, whenever $B$ is an $A$-neighborly subset of $A$, there is at most one element of $A$ in each face of the bonnet over $B$.\end{definition}

\begin{lemma}
\label{facelemma}
Suppose $A$ is generic and $c-\interior\Orth^n$ contains no points of $A$.  Then there is at most one element of $A$ in each face of $c- \Orth^n$. \end{lemma} 

\begin{proof}  The elements of $A$ that lie in any face of  $c-\Orth^n$ all lie in the corresponding face of the bonnet over $A\cap (c-\Orth^n)$.
 Indeed, suppose $B=A\cap (c-\Orth^n)$.  If $B$ is empty, there is nothing to prove.  Otherwise, let $b=\vee B$. If  $a\in B$ lies in $c-\Orth^n_J$, then $\pi_j(a) = \pi_j(c)$ for all $j\in J$.  Since $a\leq b\leq c$, $a$ is contained in $b-\Orth^n_J$.  \end{proof}

\begin{lemma}   The properties of $A$-neighborliness and of being generic are translation-invariant in the following sense: if $B$ is $A$-neighborly and $x\in \R^n$, then $x+B$ is $x+A$-neighborly, and if $A$ is generic, then so is $x+A$.\end{lemma}  

\begin{proof}  This is immediate from the definitions and the translation invariance of $\leq$.\end{proof} 

\begin{lemma}
\label{genericlemma}
The following are equivalent:\begin{enumerate}
\item[$i)$] $A\subseteq \R^n$ is generic.
\item[$ii)$] If  $x, y$ are distinct elements of $A$ and $\pi_i(x) = \pi_i(y)$ for some $i\in \{1,2,\ldots,n\}$, then there is $z\in A$ with $z<<x\vee y$.
\end{enumerate}
\end{lemma}  

\begin{proof}   $i)\Rightarrow ii)$.  Assume $i)$ and assume $x, y\in A$ and $\pi_i(x) = \pi_i(y)$.  Then, $x$ and $y$ are not $A$-neighbors because they lie on the same face of the bonnet over $\{x,y\}$.  Therefore, there is an element $z\in A$ in the interior of the bonnet over $\{x,y\}$, so $z<<x\vee y$.
$ii)\Rightarrow i)$.  Let $B$  be a bounded subset of $A$.  We must show, using $ii)$, that if there are distinct points of $A$ lying on the same face of the bonnet over $B$, then $B$ is not neighborly.  Let $x$ and  $y$ be such points.  Now, $x, y\leq \vee B$ and for some $i$,  $\pi_i(x) = \pi_i(\vee B)=\pi_i(y)$.  By $ii)$, there is $z<< x\vee y\leq \vee B$, so $B$ is not neighborly. \end{proof} 

\begin{lemma}
\label{codimlemma}
Suppose $A$ is generic and $B\subseteq A$ is $A$-neighborly.  Then: \begin{enumerate}
\item[$i)$]   There is exactly one element of $B$ in each codimension-one face of the bonnet over $B$ (so, the cardinality of $B$ is at most $n$). 
\item[$ii)$]  The only elements of $A$ in the bonnet over $B$ are the elements of $B$ itself.
\item[$iii)$]   The sum of the codimensions of the minimal faces of the bonnet over $B$ that contain elements of $B$ is exactly $n$.
\item[$iv)$]  If $B'\subseteq A$ is also $A$-neighborly and $\vee B = \vee B'$, then $B=B'$
\end{enumerate}\end{lemma}

\begin{proof}  Addressing $i)$, there must be at least one element of $B$ in each face of the bonnet over $B$ even if $B$ is not generic, and if $B$ is generic, then by definition there is at most one.  Assertion $ii)$ is immediate from $i)$.  For $iii)$,  suppose  $B = \{y_1, \ldots,y_m\}$, $m\leq n$. The sets $$J_i:=\{\,j\mid \pi_j(y_i) = \pi_j(\vee B)\,\}$$ are disjoint and their union is $\{1,\ldots,n\}$, because each coordinate of $\vee B$ is determined by one of the $y_i$ in $B$.  For assertion $iv)$, suppose $\vee B = \vee B'$.  Then, $B\cup B'$ is $A$-neighborly.  Now by $i)$, each codimension-one face contains exactly one element of $B$, exactly one element of $B'$ and exactly one element of $B\cup B'$.  So, $B=B'$.   \end{proof} 

\begin{remark} The definitions and lemmas  in this section concerning $\Nb(A)$ and its properties generalize to antichains in $[-\infty,+\infty\,]^n$, since this set is isomorphic as an ordered set to the subset $[-1,1]^n$ of $\R^n$.  The material in subsequent sections also generalizes, since an order-isomorphism  $[-\infty,+\infty\,]\to[-1,1]$ (such as $(2/\pi)\arctan$) is also a topological equivalence.   Points with coordinates in $\{\pm\infty\}$ are referred to in \cite{Sc} as ``slack vectors''. \end{remark} 

\section{Weak $A$-neighbors and local finiteness}

In this section, we show that if $A$ is any generic subset of $\Z^n$, then $\Nb(A)$ has the property that every vertex belongs to at most finitely many simplices.  It is enough the show that every element of $A$ has finitely many $A$-neighbors, which is what Theorem \ref{locfin} asserts.  We use Dickson's Lemma to prove this, and as a bonus, we include a very simple proof of it. 
 
\begin{definition}  Suppose  $x, y\in \R^n$. The set $$\{\,z\in \R^n\mid \hbox{for $i=1,2, \ldots, n$, $\pi_i(z)$ is in the closed interval from $\pi_i(x)$ to $\pi_i(y)$}\,\}$$ is called the {\it rectangle from\/ $x$ to\/ $y$\/}. Suppose $A\subseteq \R^n$ and $y\in \R^n$.  We say that $x\in A$ is a {\it weak $A$-neighbor of $y$\/} if $x\not=y$ and there is no $z\in A$ other than $x$ and $y$ in the rectangle from\/ $x$ to\/ $y$.\end{definition}

Weak $A$-neighborliness is translation-invariant in the following sense:
 $x$ is a weak $A$-neighbor of $y$ if and only if $x-y$  is a weak $(A-y)$-neighbor of $0$.  

\begin{remark}  The concept of weak $A$-neighbor is a generalization of the idea of a ``primitive lattice vector'' that appears in numerous works of B.~Sturmfels.  In \cite[p.\kern1.5pt 33]{St}, a vector $\alpha$ in a sublattice $L\subseteq \Z^n$ is defined to be primitive if there is no $\delta\in L$ other than $0$ and $\alpha$ such that $\delta^+\leq \alpha^+$ and $\delta^-\leq \alpha^-$.  Evidently, $\alpha\in L$ is primitive if and only if it is a weak $L$-neighbor of $0$. \end{remark}

\begin{lemma}Suppose $A\subseteq \R^n$ is generic and $x,y\in A$.  If $x$ is an $A$-neighbor of $y$, then $x$ is a weak $A$-neighbor of $y$.\end{lemma}

\begin{proof} Suppose $x\in A$ is not a weak $A$-neighbor of $y$.  Pick $z\in A$ other than $x$ and $y$ such that: for all $i\in \{1,2,\ldots,n\}$,  $\pi_i(z)$ is in the closed interval from $\pi_i(x)$ to $\pi_i(y)$ .  Then $(z-y)^+\leq (x-y)^+$ and $(z-y)^-\leq (x-y)^-$ .  It follows that $z- y\leq (x- y)\vee 0$, so $z\leq x\vee y$.  By Lemma 2.5.$ii$, $x$ and $y$ are not $A$-neighbors.\end{proof}

\begin{lemma} 
\label{wlocfin} 
Suppose $A\subseteq \Z^n$ and $\beta\in \Z^n$.  Then $\beta$ has at most a finite number of weak $A$-neighbors.\end{lemma}

\begin{proof}  Let $\Delta$ is a diagonal matrix all of whose diagonal entries are in $\{1, -1\}$.  A set of the form $\Delta \N^n$ is called an {\it orthant of $\Z^n$\/}.   Define a partial order $\leq_\Delta$ on $\Delta \N^n$  by $$\alpha\leq_\Delta\beta \;:\Longleftrightarrow\; \Delta\alpha\leq  \Delta\beta \;\Longleftrightarrow\;  \alpha^+\leq \beta^+ \;\&\;\alpha^-\leq \beta^-.$$  Then $\Delta \N^n$ is order-isomorphic to $\N^n$.  Moreover, $\alpha\in \Delta\N^n$ is a weak $A$-neighbor of $0$ if and only if $\alpha$ is  $\leq_\Delta$-minimal in $A\cap\Delta\N^n\setminus\{0\}$.  Applying Dickson's Lemma (see below) to each orthant, it follows that for any $A\subseteq\Z^n$, $0$ has at most finitely many weak $A$-neighbors.  The general result follows by translation invariance of $\leq_\Delta$.\end{proof}

The following is an immediate consequence of the last two lemmas.

\begin{theorem}
\label{locfin}  Suppose $A\subseteq \Z^n$ is generic and $\beta\in A$.  Then ${\rm nbr}_A(\beta)$ is finite. \end{theorem}


In the proof of Lemma \ref{wlocfin}, we used Dickson's Lemma.  Many proofs of this  have appeared in the literature.  Below, we present a particularly quick and transparent proof that does not seem as well-known as it deserves to be.  We say that a sequence $\{\,\alpha_i\mid i\in \N\,\}$ of elements $\alpha_i$ of some poset is {\it  weakly increasing\/} if $\alpha_i\leq \alpha_{i+1}$ for all $i$.  By a {\it subsequence\/} of $\{\,\alpha_i\mid i\in \N\,\}$, we mean a sequence  $\{\,\alpha_{s(i)}\mid i\in \N\,\}$  determined by a strictly increasing function $s:\N\to\N$.    Note that every sequence of elements of $\N$ has a weakly increasing subsequence, since any unbounded sequence in $\N$ contains a strictly increasing subsequence and any bounded sequence in $\N$ contains a constant subsequence.

\begin{lemma} [Dickson's Lemma]  Every sequence 
in $\N^n$ contains a weakly increasing subsequence.  Thus, $\N^n$ contains no infinite antichains.  In particular, the set of $\leq$-minimal elements in any subset of $\N^n$ is finite.\end{lemma}

\begin{proof}  Any sequence of elements of $\N^n$
contains a subsequence in which the last coordinate
is weakly increasing.  This, in turn, contains a subsequence in
which the $(n-1)$th coordinate is weakly increasing.  After $n$
steps, we have a subsequence satisfying the required condition. \end{proof} 

\section{Step hypersurfaces}

In \cite{Sc}, Scarf alludes to hypersurface that we examine in this section, but does not study its structure in any detail.  In \cite[Definition 3.6]{MS}, the authors define the ``staircase surface'' of a monomial ideal $I$ in $k[\N^3]$. This is the $n=3$ case of the object analyzed in the present section.  The ``Buchberger graph of $I$,''  \cite[Definition 3.4]{MS}, is the one-skeleton of the Scarf complex $\Nb(G_I)$.  As seen in illustration following 3.4, the Buchberger graph has additional structure: each edge contains the supremum of the vertices that it connects.  In fact, with these points taken into account, Buchberger graph is the  barycentric subdivision of the one-skeleton of $\Nb(A)$.   In the present section, we give a construction that generalizes the staircase surface to any dimension and displays the Buchberger graph as a special instance of a general construction involving barycentric subdivision that yields a concrete triangulation of the generalized staircase surface.  At the end of the section, we give some examples showing that $\Nb(A)$ may have complicated---in particular, non-simply-connected---topology when $A$ is infinite.  In the section following, we establish conditions that preclude such behavior.

\begin{lemma}
\label{bdlemma}  
Let $D$ be a down-set in $\R^n$ and let $\partial D$ denote its topological boundary.  Then, $$b\in \partial D\quad\Longleftrightarrow\quad (b+\interior\Orth^n)\subseteq (\R^n\setminus D)\;\;\hbox{and}\;\;b-\interior\Orth^n\subseteq D.$$ \end{lemma}

\begin{proof} By definition,  $b\in \partial D$ iff every open box centered on $b$ contains at least one point of $D$ and at least one point not in $D$; the implication $\Leftarrow$ follows.  To prove $\Rightarrow$, consider the contrapositive.  If  $z\in D\cap(b+\interior\Orth^n)$, then $b\in z-\interior\Orth^n\subseteq D$, so $b \not\in \partial D$.  If $z\in b-\interior\Orth^n\setminus D$, then $b\in z+\interior\Orth^n$, and $z+\interior\Orth^n\cap D=\emptyset$, so $b \not\in \partial D$.  \end{proof} 

\begin{definition}  Let $A\subseteq \R^n$.  Then $m A$ denotes the set of all suprema of maximal $A$-neighborly subsets of $A$, $dmA:=mA-\Orth^n$ and  $\partial dmA$ denotes the topological boundary of $dmA \subseteq \R^n$. \end{definition}  

{\it Throughout the remainder of this section, $A$ is assumed to be a generic antichain in $\R^n$.}    Since $A$ is an antichain, every singleton subset of $A$ is $A$-neighborly and therefore, since $A$ is generic, every element of $A$ is contained in a maximal $A$-neighborly subset of $A$ (since an $A$-neighborly set has at most $n$ elements, by Lemma \ref{codimlemma}). This implies that $A\subseteq dmA$.

\begin{lemma} If $B\subseteq A$ is $A$-neighborly, then $\vee B\in \partial dmA$.  In particular, $A\subseteq \partial dmA$.  \end{lemma} 

\begin{proof}  Let $B\subseteq A$  be $A$-neighborly with supremum $b$.  Since $B$ is contained a maximal $A$-neighborly subset of $A$,  we have $b\in dmA$.  Let $x\in mA$.  Then $(b+\interior\Orth^n)\cap (x-\Orth^n) =\emptyset$ because if not, then for some $b\in B$ and some $0<<p\in \R^n$, $b+p \leq x$, so $b\leq x-p$, so $b$ lies in $x-\interior\Orth^n$ contrary to the assumption that $x$ is the supremum of an $A$-neighborly set.  Since $dmA=\cup\{\,x-\Orth^n\mid x\in mA\,\}$, $b+\interior\Orth^n\subseteq (\R^n\setminus dmA)$.  It follows from \ref{bdlemma} that  $b\in \partial dmA$.\end{proof} 

Suppose $\Delta$ is an abstract simplicial complex. Recall that the {\it abstract barycentric subdivision of $\Delta$\/}, which we here denote $\sd \Delta$, is constructed as follows.  The vertices of $\sd \Delta$ are in bijection with the simplices of $\Delta$, and if $\sigma$ is a simplex of $\Delta$, the corresponding vertex of $\sd \Delta$ will be denoted  $\dvee\sigma$.  Now, suppose $\sigma\in\Delta$.  Let $\pi$ be a total ordering of $\sigma$.  For $j=1,2,\ldots, m$ let $\sigma^\pi_j$ denote the set consisting of the first $j$ elements of $\sigma$ with respect to the ordering $\pi$.   Then, we declare $s(\sigma, \pi):=\{\dvee\sigma^\pi_1, \dvee\sigma^\pi_2,\ldots, \dvee\sigma^\pi_m\}$ to be a simplex of $\sd \Delta$, and of course each its subsets as well.  In general,  the simplices of $\sd \Delta$ are the subsets of the vertex set  created in this manner.  In particular, each $m$-dimensional simplex of $\Delta$ gives rise to $(m+1)!$  simplices  of dimension $m$ in $\sd \Delta$, as well as to the sub-simplices of these. 

Let us apply the construction in the previous paragraph to $\Nb(A)$.  Suppose $B\in N_{m-1}(A)$.  Let $\pi$ be a total ordering of $B$, and assume the elements of $B$ written in this order are $(b_1,\ldots, b_m)$.  In the notation above, $B^\pi_j = \{b_1,\ldots,b_j\}$.   Now, let $c_j=\bigvee_{i=1}^j b_i =\vee B^\pi_j\in \R^n$.   (Here, $\vee$ refers to the supremum operation in $\R^n$.) Let $C(B, \pi)$ denote the convex hull of $\{c_1, \ldots, c_m\}$.

\begin{lemma}
\label{simplexlemma}
The points $c_1, \ldots, c_m\in \R^n$ are affinely independent. The geometric simplex $C(B, \pi)$ is contained in $\partial dmA$.
\end{lemma}

\begin{proof} Referring to the notation in the lemma, note that using Lemma \ref{codimlemma}, part $i)$, and re-ordering the coordinates if necessary, we may assume that there are integers $0=n_0<n_1<n_2<\cdots<n_m= n$ such that  that  $$\hbox{for $j=1,2,\ldots, m$: $\pi_\alpha (c_m-b_j) = 0\;\Leftrightarrow\;n_{j-1}<\alpha\leq n_j$},$$  and hence  $$\hbox{for $j=1,2,\ldots, m$: $\pi_\alpha (c_m-c_j) = 0\;\Leftrightarrow\;1\leq\alpha\leq n_j$}.$$  This shows affine independence.   Now, suppose $c\in C(B,\pi)$.  Then $c=r_1c_1+\cdots+ r_m c_m$ with $r_i\in [0,1]$ and $r_1+\cdots +r_m=1$.  We have $b_1\leq c$, so $c+\interior\Orth^n\subseteq b_1+\interior\Orth^n\subseteq (\R^n\setminus dmA)$.  On the other hand, $c\leq c_m$, and $c_m\in dmA$, so $c-\interior\Orth^n\subseteq dmA$.\end{proof}

\begin{lemma} 
\label{uniquenesslemma} 
Suppose $B,B'\in N_{m-1}(A)$.  Let $\pi$ and $\pi'$ be total orderings of $B$ and $B'$, respectively.  If $C(B,\pi)=C(B',\pi')$, then $B=B'$ and $\pi=\pi'$.\end{lemma}

\begin{proof} Let $c_1,\ldots,c_m$ and $c'_1, \ldots, c_m$ be constructed as in Lemma \ref{simplexlemma}.  By Lemma \ref{codimlemma}, part $iv)$, $c_i = c'_{j}$ if and only if $i=j$ and $(b'_1, \ldots,b'_i)$ is a permutation of $(b_1, \ldots, b_i)$.\end{proof}

\begin{definition} Let $\cd \Nb(A)$ denote the set consisting of the geometric simplices $C(B, \pi)$ for $B\in \Nb(A)$ and $\pi$ and ordering of $B$, as well as all the subsimplices of the $C(B, \pi)$. \end{definition} 

The union of the simplices in $\cd \Nb(A)$ is contained in $\partial dmA$.  Now, Lemmas \ref{simplexlemma} and \ref{uniquenesslemma} show that 
$$\sd \Nb(A)\ni (\sigma, \pi)\;\leftrightarrow C(\sigma, \pi)\in \cd \Nb(A)$$
is a bijection.  The following is an immediate consequence: 

\begin{theorem}  As an abstract simplicial complex on the vertex set $\{\,\vee B\mid B\in \Nb(A)\,\}$, $\cd \Nb(A)$ is equivalent to $\sd \Nb(A)$.  Thus, $\cd \Nb(A)$ is a geometric realization of $\sd \Nb(A)$ contained in $\partial dmA$.\end{theorem}

Let $D$ be a proper down-set in $\R^n$ (i.e., $D\not=\emptyset$ and $D\not= \R^n$) with topological boundary $\partial D$.  Let $\ell$ be any line parallel to a vector $p\in \interior\Orth^n$.  Then $\ell$ contains points of $D$ and points not in $D$ as well, since $\ell$ meets $x+\Orth^n$ and $y-\Orth^n$ for any $x,y\in\R^n$.  Moreover, $\ell$ meets $\partial D$ in a unique point, $\bigvee(\ell\cap D)$.  Thus, if $H$ is a hyperplane of dimension $n-1$ perpendicular to $p$,  then projection parallel to $p$ gives a bijection of $\partial D$ onto $H$. Evidently, $Y$ is open in  $\partial D$ if and only if the projection of $Y$ is open in $H$.   Under this projection, each $d$-simplex of $\cd \Nb(A)$ is taken to a $d$-simplex in $H$, so we can see that $\cd \Nb(A)$ is PL equivalent to a (not necessarily compact) PL subset of $\R^{n-1}$.   

\begin{example}What hypotheses on $A$ are needed to assure that $\cd \Nb(A)$ contractible?  The following example shows that it is not adequate to assume that $A$ is a discrete, generic antichain. Consider concentric circles about $(0,0,0)$ of radii $1+1/i$ ($i=1,2,\ldots$), all lying in the plane $H$ defined by $x+y+z=0$.  On the $i^{th}$ circle, choose at least $i$ points spaced  evenly up to a very small error and placed so that no new point lies on any of the lines $x=k$, $y=k$ or $z=k$ ($k$ any constant) passing through any of the previously chosen  points on this or any larger circle.  Let $A$ be the set of all such points.  Three points of $A$ are $A$-neighborly if they lie on the boundary of an $A$-free triangle with sides parallel to the lines $x=0$, $y=0$ and $z=0$.  But a triangle with vertices from $A$ and with sides parallel to these lines is not $A$-free if one of the edges meets the closed unit disk in $H$ about $(0,0,0)$.  Thus, if we project $\cd \Nb(A)$ onto $H$, the image covers an annular region outside the closed unit disk, but it omits the disk itself.  This example can be generalized.  Let $H$ be the hyperplane perpendicular to $(1,1,\ldots,1)$ in $\R^n$, and let $U$ be any open subset of $H$.  Let $A$ be a discrete set of points in $U$, such that every open ball about any boundary point of $U$ contains a point of $A$.  (With some care, we may choose the points of $A$  so that it is generic.)  If $B$ is an $A$-neighborly subset of $A$ then, $H\cap \vee B-\Orth^n$ must be contained in $U$, because any polygonal subset of $H$ with non-empty interior that contains a point in $U$ and a point not in $U$ must have points of $A$ in its interior.  \end{example}

\begin{remark} Note that if the points of $A$ all lie on a line in a hyperplane $H$ perpendicular to some vector in $\interior\Orth^n$, $n\geq 3$, then $\cd \Nb(A)$ is 1-dimensional, and its projection onto $H$ is certainly not all of $H$.  Can we find a closed, discrete, generic antichain $A$ such that $\cd \Nb(A)$ is one-dimensional but does not lie on a line?  Can we arrange for it to be a tree with vertices of valence greater than 2? \end{remark}

\section{Contractibility of $\Nb(A)_{\leq b}$}

\begin{definition}  For $b\in \R^n$, let $\Nb(A)_{\leq b}:=\{\,B\in \Nb(A)\mid \vee B\leq b\,\}$ and $\Nb(A)_{< b}:=\{\,B\in \Nb(A)\mid \vee B< b\,\}$.\end{definition} 

Note that $\Nb(A)_{\leq b} = \Nb(A_{\leq b})$ and $\Nb(A)_{< b} = \Nb(A_{< b})$, since the elements of $\Nb(A)$ are finite subsets of $A$.   Applications of $\Nb(A)$ to minimal free resolutions depend critically on the fact that if $A$ is finite, then $\Nb(A)$ is contractible.  For more on this, see \cite[Proposition 4.5]{MS}.

Two proofs of contractibility have appeared in the literature.  The first, which originates in \cite{BHS}, uses the exponential map.  Consider the family of functions, $E_t:\R^n\to \R^n$ parametrized by positive real numbers and defined by the condition $\pi_i(E_t(x)) = t^{\pi_i(x)}$. For $X\subseteq \R^n$, let  $\conv X$  denote the convex hull of $X$, and let $P_t(A):=\conv (E_t(A))+\Orth^n$.  If $A$ is finite, then $P_t(A)$ is a polyhedron.  The  same arguments used to prove Proposition 4.14 and Theorem 4.17 of \cite{MS}  show that if $A$ is any finite antichain in $\Orth^n$, then there is $t_0\in \R$ such that if $t>t_0$ then the vertices of $P_t(A)$ are the points $\{E_t(a)\mid a\in A\}$ and  face complex of $P_t(A)$ is independent of $t$.  The {\it hull complex\/} of $A$ is, by definition, the cell complex of bounded faces of $P_t(A)$, for $t>t_0$.  Theorem 6.13 of \cite{MS} shows that when $A$ is generic, $\Nb(A)$ is equivalent to the hull complex of $A$, and Theorem 4.17 then shows that  $\Nb(A)_{\leq b}$ is contractible for any $b\in \Orth^n$.  The proof of 4.17 is based on a lemma from polyhedral topology that says that if $P$ is a polyhedron and $F$ is a face of $P$, then the complex of faces of $P$ disjoint from $F$ is contractible.  The theorem follows from the fact that $\Nb(A)_{\leq b}$ can be identified with the faces of $P_t(A)$ that lie on one side of a hyperplane whose position depends on $b$.   

The second approach, due to Olteanu and Welker \cite{OW}, uses combinatorial poset homotopy.  Section 10 of \cite{B} contains a useful synopsis of this theory.    With any poset $P$ we associate the abstract simplicial complex $\Delta(P)$, whose vertices are the elements of $P$ and whose simplices are the chains in $P$.  We say $P$ has a topological property (such as contractibility) when  $\Delta(P)$ has that property.  For example, if $P$ has a largest or a smallest element $p$, then $P$ is contractible since $\Delta(P)$ is a cone over $p$.  The Quillen Fiber Lemma  \cite[10.5]{B} says that if  $f: Q \to P$ is a poset map such that  $f^{-1} (P_{\leq p})$ is contractible for all $p \in P$, then $Q$ and $P$ are homotopy equivalent.  An antichain $C\subseteq P$ is called a {\it crosscut\/} if (a) every chain in $P$ is contained in a chain that meets $C$ and (b) every bounded subset of $C$ (i.e., set with either an upper or a lower bound in $P$) has either a supremum or infimum in $P$.  If $C$ is a crosscut in $P$, $\Gamma(P,C)$ denotes the simplicial complex consisting of the bounded subsets of $C$.  The Crosscut Theorem \cite[10.8]{B} says that $\Gamma(P,C)$ and $P$ are homotopy equivalent.

The lemmas and propositions below are streamlined (and slightly generalized, since we assume only that $A\subseteq \R^n$) versions of material from \cite{OW}.   {\it In all of them,  $A$ is assumed to be a\/ {\bf finite\/} antichain in $\R^n$.}  We do {\it not\/} assume that $A$ is generic.   Let $$L(A):=\{\,\vee B\mid B\subseteq A\,\}.$$  Note that the elements of $A$ are minimal in $L(A)$.  Let $$P(A):=\{\, b\in L(A)\mid \hbox{there is {\bf no} $a\in A$ with $a<<b$}\,\}.$$  Then $P(A)$ is a down-set in $L(A)$, and $A\subseteq P(A)$.

\begin{lemma}  Let $b\in L(A)\setminus P(A)$.  Then $L(A)_{<b}$ is contractible. \end{lemma} 

\begin{proof} Pick $a\in A$ such that $a<<b$.  The map $u\mapsto u\vee a:L(A)_{<b} \to [a, b)=[a, b)_{L(A)}$ preserves order and and satisfies $u\leq u\vee a$ for all $u$, so $L(A)_{<b}$ and $[a,b)$ are homotopy equivalent by \cite[10.12]{B}.  Since $[a, b)$ has a smallest element, it is contractible.  \end{proof}

\begin{lemma}   Let $P$ be a poset with maximal element $m$ such that $P_{<m}$ is contractible.  Then $P$ and $P \setminus \{m\}$ are homotopy equivalent.  \end{lemma}  

\begin{proof}   We apply the Quillen Fiber Lemma. Let $Q := P \setminus \{m\}$.  Consider the natural inclusion $i: Q \subset P$ (so $i^{-1}(X)=X\cap Q$). Let $p\in P$.  If $p\not=m$, $i^{-1}(P_{\leq p})= P_{\leq p}$, which is contractible since it has a largest element. On the other hand, $i^{-1}(P_{\leq m}) = Q_{\leq m}= P_{< m}$ is contractible by assumption.  The lemma follows.\end{proof}

\begin{proposition}  $P(A)$ is contractible.\end{proposition}

\begin{proof} $L(A)$ has a unique maximal element $\vee A$, so it is contractible.   Let $X^{(0)}$ = $L(A)$.  Now we construct $X^{(i)}$, $i=1,2,\ldots$ by induction. If $X^{(i)}$ has been defined and is not equal to $P(A)$, then $X^{(i)}$ has at least one maximal element  that is not in $P(A)$.  Let $m_i$ be one such element and  define $X^{(i+1)}:= X^{(i)}\setminus\{m_i\}$.  By Lemmas 1 and 2, all the $X^{(i)}$ are all contractible.    Since $L(A)$ is finite, for some $i$,  $X^{(i)}=P(A)$.\end{proof}

\begin{proposition}  $A$ is a crosscut in $P(A)$, and $\Gamma(P(A),A) = \Nb(A)$.  Thus, $\Nb(A)$  is contractible.\end{proposition} 

\begin{proof} The elements of $A$ are the minimal elements of $P(A)$, so condition (a) for a crosscut is satisfied.  By definition of $P(A)$, the subsets of $A$ that are bounded in $P(A)$ are the $A$-neighborly subsets of $A$  and $P(A)$ consists of the suprema of such sets, so condition (b) is satisfied.  This also shows that $\Gamma(P(A),A) = \Nb(A)$.  The second statement now follows immediately from the Crosscut Theorem. \end{proof}

\begin{corollary}  If $A$ is a (possibly infinite) antichain in $\R^n$,  $b\in \R^n$, and $A_{\leq b}$ (respectively $A_{<b}$) is finite, then $\Nb(A)_{\leq b}$ (respectively $\Nb(A)_{<b}$) is contractible. \end{corollary}

\section{Global Topology of $\Nb(A)$}
We will add to $A\subseteq \R^n$ certain ideal points with infinite coordinates and then extend the definition of $\Nb(A)$ accordingly, allowing a bonnet $b-\Orth^n$ to be defined by a $b$ that has some infinite coordinates, i.e., $b\in [-\infty,+\infty\,]^n$.  This idea comes from \cite{Sc}, where the ideal points are called ``slack vectors,'' and it is used in \cite{BSS} to prove that a complex closely related to $\Nb(A)$ is a triangulation of $\R^{n-1}$ when $A$ is a lattice.  We prove a similar result for any $A$ such that $A^\ast$ is generic.  In \cite{BSS}, the proof is based on the exponential map.  Instead, we use the facts about $\partial dnA^\ast$ that we established in \S 5, above.

\begin{definition} Let $w_i$, $i=1, \ldots,n$ be defined by  $$\pi_j(w_i) =\Big\{\begin{array}{lr}
+\infty,& \text{if }i=j;\\ 
-\infty,& \text{if } i\not=j.\end{array}$$  If $A\subseteq \R^n$, let $A^\ast:=A\cup \{w_1, \ldots, w_n\}$.  \end{definition}

Let $W$ be the combinatorial $(n-1)$-simplex on the vertex set $\{w_1, \ldots, w_n\}$.  We may identify the interior of $|W|$ with the hyperplane $H$ in $\R^n$ that contains the origin and is perpendicular to $\mathbf{1}=(1,1,\ldots,1)$.  In this picture, the boundary of $|W|$ is an $(n-2)$-sphere ``at infinity,'' that compactifies $H$.  The sub-simplices of $W$ are realized as subsets of this $(n-2)$-sphere.

\begin{proposition}\label{prop7}  Let $A\subseteq \R^n$.  Assume that  the projection $\pi_i(A)$ is closed and discrete for each $i=1,2,\ldots,n$, and that $A^\ast$ is a generic antichain in $[-\infty, +\infty]^n$.  Let $b\in \R^n$ and suppose that there is no $a\in A^\ast$  with $a\leq b$.  Then, there is $B\in N_{n-1}(A^\ast)$ such that $b< \vee B$.  \end{proposition}

\begin{proof} Let $\{e_1, \ldots, e_n\}$ be the standard basis for $\R^n$, i.e., $\pi_i(e_i)=1$ and for $j\not=i$, $\pi_j(e_i)=0$.  Assume that $b$ satisfies the hypotheses of the lemma.  Define $b_1$ as follows:
\begin{enumerate}
\item[$i)$]  If there is $\lambda \in \R_{\geq 0}$ such that $A_{\leq b+\lambda e_1}$ is nonempty, then (because $\pi_1(A)$ is closed and discrete)  there is a smallest such $\lambda$, call it $\lambda_1$.  By assumption on $b$, $\lambda_1>0$.  Let $b_1:=b+\lambda_1 e_1$.   
\item[$ii)$] Otherwise let $b_1$ be defined by $\pi_1(b_1) = +\infty$ and for $i>1$, $\pi_i(b_1) = \pi_i(b)$.  
\end{enumerate}
Because $A_{\leq b}$ is empty and $b_1$ differs from $b$ only the first coordinate,  $A_{\leq b_1}$ is contained in the face $b_1-\Orth^n_{\{1\}}$ of $b_1-\Orth^n$.   By Lemma \ref{codimlemma}, $b_1-\Orth^n_{\{1\}}$ contains {\it only\/} one point of $A^\ast$, call it $a_1$.  Note that $a_1$ might be $w_1$.  Now we continue the process.  Define $b_2$ by increasing the second coordinate of $b_1$  (possibly to $+\infty$) just until $b_2-\Orth^n_{\{2\}}$ contains a point $a_2$ of $A^\ast$ other than $a_1$, possibly $a_2=w_2$.  (Note that $b_1-\Orth^n_{\{2\}}$ might contain  $a_1$.  If so, however, $a_1$ is not in the relative interior of $b_1-\Orth^n_{\{1\}}$, and we then increase the second coordinate of $b_1$ just until we obtain  $b_2$ such that $b_2- \Orth^n_{\{2\}}$ contains a new point of $A^\ast$.  Here, we are using the hypothesis that $\pi_2(A)$ is closed and discrete.)   Continue in this fashion.  At every step, we  properly increase the coordinate that we are adjusting.    When we have gone through all the coordinates, we have  constructed $b_n\geq b$ with the property that each codimension 1 face of $b_n-\Orth^n$ contains {\it exactly\/} one point of $A^\ast$, and there are no points of $A^\ast$ in $b_n-\interior\Orth^n$.  Observe that $b_n$ is the supremum of the points $a_1, \ldots,a_n\in  A^\ast$, and these points lie in the interiors of the codimension 1 faces of $b_n-\Orth^n$.  These points form a maximal $A$-neighborly set $B$, and $b_n-\Orth^n$ is the bonnet over that set.  Note that if $A$ is non-empty, then at least one coordinate of $b_n$ must be finite.  \end{proof}

\begin{remark}  We may modify the construction in the proof by taking the coordinates in some order other than the default order.  In principle, then, there might be as many as $n!$ different bonnets that contain $b$. 
\end{remark}

\begin{theorem} Suppose $A$ satisfies the conditions in Proposition \ref{prop7}.  Then $|\Nb(A*)|\setminus|W|$ is homeomorphic to $\R^{n-1}$.\end{theorem}

\begin{proof}By the proposition, every line in $\R^n$ parallel to $\mathbf{1}$ meets $|\cd\Nb(A^\ast)|$, so the finite part of this set is homeomorphic to $\R^{n-1}$.  As we have shown above, $\cd\Nb(A^\ast)$ is equivalent to the barycentric subdivision of    $\Nb(A*)$. \end{proof}

\end{document}